\documentclass{amsart}

\usepackage{amsmath,amssymb,amsthm,amscd}
\newtheorem{theorem}{Theorem}[section]
\newtheorem{lemma}{Lemma}[section]
\newtheorem{proposition}{Proposition}[section]
\newtheorem{corollary}{Corollary}[section]
\newtheorem{definition}{Definition}[section]
\newtheorem{remark}{Remark}[section]
\newtheorem{example}{Example}[section]

\begin{document}
\title{$\Sigma$-semi-compact rings and modules}
\author{M. Behboodi}\footnote{The research of the first author was in part supported by a grant from IPM (No. 92130413)}
\address{Department of Mathematical
 Sciences, Isfahan
University of Technology, P.O.Box: 84156-83111, Isfahan, Iran, and
        School of Mathematics
 Institute for Research in Fundamental Sciences (IPM)
 P. O. Box 19395-5746, Tehran, Iran
}
\email{mbehbood@cc.iut.ac.ir} 

\author{F. Couchot}
\address{Universit\'e de Caen Basse-Normandie, CNRS UMR
  6139 LMNO,
F-14032 Caen, France}
\email{francois.couchot@unicaen.fr}  

\author{S. H. Shojaee}
\address{Mazandaran University of Science and Technology, P.O.Box: 48158-78413, Behshahr, Iran.
}
\email{hshojaee@math.iut.ac.ir}



\begin{abstract} In this paper several characterizations
  of semi-compact modules are given. Among other
results, we study rings whose semi-compact modules are injective.
We introduce the property $\Sigma$-semi-compact for modules
and we characterize  the modules satisfying this property. In particular, we show that a
ring $R$ is left $\Sigma$-semi-compact if and only if $R$  satisfies
the ascending (resp. descending)  chain condition on the left
(resp. right) annulets. Moreover, we prove that every flat left
$R$-module is semi-compact if and only if  $R$ is left
$\Sigma$-semi-compact. We also show that a ring $R$
 is left Noetherian if and only if every pure projective left $R$-module
  is semi-compact. Finally, we consider rings whose
flat modules are finitely (singly) projective. For any commutative arithmetical ring $R$ with quotient ring $Q$, we prove that every flat $R$-module
is semi-compact
if and only if every flat $R$-module is finitely (singly) projective if
 and only if $Q$ is pure semisimple. A similar result is obtained for reduced commutative rings $R$ with the space $\mathrm{Min}\ R$ compact. We also prove that 
every $(\aleph_{0},1)$-flat left $R$-module is singly projective if $R$ is left
  $\Sigma$-semi-compact, and the converse holds if $R^{\mathbb{N}}$ is an $(\aleph_{0},1)$-flat left $R$-module.
\end{abstract}

\maketitle

\section*{Introduction}

We shall assume that all rings are associative with identity and
all modules are unitary. Let $R$ be a ring and $M$
be a left $R$-module.  For any left ideal $I$ of $R$, set $M[I] =
\{m\in M~|~Im = 0\}$. It is a subgroup of $M$. As in \cite{Mat59} and
\cite{Mat85}, $M$ is said to be {\it semi-compact} if every
finitely solvable set of congruences $x \equiv x_{\alpha}$ (mod
$M[I_{\alpha}]$)
 ($\alpha \in \Lambda$,  $x_{\alpha}\in M$ and  $I_{\alpha}$ is
 a left ideal of $R$ for each $\alpha \in \Lambda$) has a simultaneous
solution in $M$. Also, we say that $M$ is {\it
$\Sigma$-semi-compact} if all direct sums of copies of $M$
are semi-compact. Semi-compactness was introduced by Matlis in \cite{Mat59} (and
also in\cite{Mat85}) for modules over commutative  rings.

In the present article, we shall study semi-compact and
$\Sigma$-semi-compact mo\-dules over arbitrary rings (not
necessarily commutative). In Section~\ref{S:semicomp},  we consider some basic
properties of  semi-compact modules, their relationship to other
concepts such as injectivity and pure-injectivity, and some rings
characterized by semi-compactness. Several characterizations of
semi-compact modules are given in  Propo\-sition \ref{P:first character} and Theorem \ref{T:character}. For instance, it is shown that a left $R$-module $M$ is semi-compact if and only if every finite solvable system of equations of the form $r_jx=a_j\in M, r_j\in R$ has a global solution in $M$, if and only if, every pure extension of $M$ is cyclically pure. So, it follows that the semi-compact modules are exactly the singly pure-injective modules introduced by Azumaya in \cite{Azu87}. 
    It is easy to see that  every semi-compact
left $R$-module is
      injective if and only if   $R$ is a von-Neumann regular
      ring (see Theorem \ref{T:Von}).  In Section~\ref{S:flat}, we introduce and study $\Sigma$-semi-compact
             modules.   It is shown that a left $R$-module $M$
  is $\Sigma$-semi-compact if and only if  $M^{(\mathbb{N})}$ is semi-compact
  if and only if  $M$ satisfies the descending
chain condition  (d.c.c.) on the  subgroups of $M$ which are
   annihilators of (finitely generated) left ideals of $R$
   (see Proposition \ref{P:faith proposition} and  Theorem \ref{T:sigma semicompact}).
It is  also shown that  every pure projective left $R$-module is
     semi-compact if and only if $R$ is left Noetherian
  (see Theorem \ref{T:Noeth-pproj}).
 In Theorem \ref{T:flat is  semicompact},
we show that  a ring $R$ is left $\Sigma$-semi-compact if and only if
every flat left $R$-module is semi-compact. If $R$
is a commutative ring, we prove that each flat $R$-module is semi-compact
if the quotient ring $Q$ is Noetherian (see Proposition \ref{P:QNoe}). 

In Section~\ref{S:fproj}, for a ring $R$ we compare the following conditions:
\begin{itemize}
\item each flat left $R$-module is semi-compact;
\item each flat left $R$-module is finitely projective;
\item each flat left $R$-module is singly projective.
\end{itemize}

There are many examples of rings for which these conditions are equivalent. For instance, if $R$ is a commutative ring and $Q$ its quotient ring, and if $R$ is either arithmetical or reduced with $\mathrm{Min}\ R$ compact, it is proven that 
these conditions are satisfied if and only if $Q$ is pure-semisimple\footnote{each commutative reduced pure-semisimple ring is semisimple}. But, if $R$ is a self left FP-injective ring, the two first conditions are not equivalent: the first holds if and only if $R$ is quasi-Frobenius, and the second if and only if $R$ is left perfect. We give an example of a self FP-injective commutative perfect ring which is not quasi-Frobenius.

It is also shown that each ($\aleph_0,1$)-flat left $R$-module is singly projective if $R$ is $\Sigma$-semi-compact as left module
  (see Proposition \ref{P:(N_0,1)-flat}) and the converse holds if $R^{\mathbb{N}}$ is a ($\aleph_0,1$)-flat left $R$-module.  

In Section~\ref{S:pure-injectivity} we investigate the rings $R$ for which each semi-compact left $R$-module is pure-injective. We get only some partial results. However, if $R$ is a reduced commutative ring, then $R$ satifies this condition if and only if $R$ is von Neumann regular.

\bigskip  

\section{Semi-compact modules}\label{S:semicomp}

\begin{definition} 
 \textnormal{ Let $M$ be a left $R$-module. For any subset (subgroup)
$X$ of $M$, 
$^{\bot}X = \{r\in R~|~rX = 0\}$
 is a left ideal of $R$. The set of such left ideals will be denoted by
$\mathcal{A}_l(R, M)$. For any subset (left ideal) $X$ of $R$, $M[X] = \{m\in M~|~Xm = 0\}$
is a subgroup of $M$. The set of such subgroups  will be denoted by
$\mathcal{A}_r(R, M)$.
Since $X\mapsto {}^{\bot}X$ is an order antiisomorphism between $\mathcal{A}_r(R, M)$ and $\mathcal{A}_l(R, M)$,
then one satisfies the a.c.c. if and only if the other satisfies the d.c.c.
In the special case $M=R$,  the elements of $\mathcal{A}_r(R, R)$ (respectively
$\mathcal{A}_l(R, R)$) are called right (respectively left) annulets of $R$: in this case we denote $X^{\bot}$ the right annulet of $X$. As in \cite{Mat59} and \cite{Mat85} $M$ is said
to be {\it semi-compact} if every finitely solvable set of congruences $x \equiv x_{\alpha}$ (mod $M[I_{\alpha}]$)
 (where $\alpha\in \Lambda$,  $x_{\alpha}\in M$ and  $I_{\alpha}$ is a left ideal of $R$ for each $\alpha \in \Lambda$) has a simultaneous solution in $M$.}
\end{definition}

\begin{lemma} \label{L:M[I][J]}
Let $M$ be a left $R$-module and $I$ and $J$ left ideals of $R$. Then \[(M[I])[J]=(M[J])[I]=M[I+J]=M[I]\cap M[J].\]
\end{lemma}

\begin{proposition} \label{P:submodule}
Let $R$ be a  ring and  $M$  a semi-compact left $R$-module. Then $M[I]$ and $M/M[I]$ are semi-compact for each  two-sided ideal $I$ of $R$.
\end{proposition}

\begin{proof}
Let  $x \equiv x_{\alpha}$ (mod $(M[I])[J_{\alpha}]=M[I+J_{\alpha}]$) be a finitely solvable set of congruences
(where $\alpha\in \Lambda$,  $x_{\alpha}\in M[I]$ and  $J_{\alpha}$ is a left  ideal of $R$ for each $\alpha \in \Lambda$).
Since $M$ is semi-compact, there exists  $m$ in $M$ such $m-x_{\alpha}\in M[I+J_{\alpha}]\subseteq M[I]$ for each $\alpha \in \Lambda$.
Since $x_{\alpha}\in M[I]$, so is $m$. Therefore, $M[I]$ is semi-compact. Now, let  $x \equiv x_{\alpha}+M[I]$ (mod $(M/M[I])[J_{\alpha}]$) be a finitely solvable set of congruences where $\alpha\in \Lambda$,  $x_{\alpha}\in M$ and  $J_{\alpha}$ is an  ideal of $R$ for each $\alpha \in \Lambda$. Obviously,
$x \equiv x_{\alpha}$ (mod $(M[IJ_{\alpha}]$) is a finitely solvable set of congruences and so has a global solution $m\in M$.
Thus $J_{\alpha}(m-x_{\alpha})\in M[I]$. Therefore, $m+M[I]-x_{\alpha}+M[I]\in (M/M[I])[J_{\alpha}]$ for each $\alpha \in \Lambda$.
\end{proof}

Let $M$ be a left $R$-module. Then the system of equations $\sum_{i\in I} r_{ij}x_i =m_j\in M$, $j\in J$ is called {\it compatible}
if, for any choice of  $s_j\in R$, $j\in J$, where only a finite
number of $s_j$ are nonzero, the relations $\sum_{j\in J}s_jr_{ij}=0$ for
each $i\in I$ imply that $\sum_{j\in J}s_jm_{j}=0$ (see \cite[Chapter 18]{Dau94} for more details about systems of equations). Throughout this
paper, all systems of equations are assumed to be compatible. The following proposition is crucial in our investigation.
\begin{proposition} \label{P:first character}
 Let $M$ be a left $R$-module. Then the following statements are equivalent: 
\begin{enumerate}
\item $M$ is semi-compact;
\item Every finitely solvable set of congruences $x\equiv x_{\alpha}$ (mod $M[I_{\alpha}]$)
(where $\alpha\in \Lambda,~x_{\alpha}\in M$ and   $I_{\alpha}$  is a finitely generated left ideal of $R$ for each $\alpha \in \Lambda$) has a simultaneous
solution in $M$;
\item Every finitely solvable  system of equations of the form
$r_jx=m_j\in M, j\in J, r_j\in R$, is solvable in $M$.
\item For each left ideal $I$ of $R$, every homomorphism $h:I\rightarrow M$, for which the restriction to any finitely  generated left subideal $I_0$ of $I$ can be extended
 to $R$,  extends itself to a homomorphism $R\rightarrow M$.
\end{enumerate}
\end{proposition}
\begin{proof}
(1)$\Rightarrow $(2) is clear.

(2)$\Rightarrow $(3). Let $r_jx=m_j\in M, j\in J, r_j\in R$, be a finitely solvable system of equations.
For each finite subset $J_{\alpha}$ of $J$, let $I_{\alpha}$ be the left ideal generated by $\{r_j\mid j\in J_{\alpha}\}$. Let
$x_{\alpha}\in M$ be a solution of  the finite system of equations  $r_jx=m_j\in M, j\in J_{\alpha}$. Obviously, the set of congruences $x \equiv x_{\alpha}$ (mod $M[I_{\alpha}]$) is finitely solvable and by hypothesis has a global solution $z\in M$. Therefore, $z$ is a solution of the above system of equations.

(3)$\Rightarrow$(1). Let $x\equiv x_{\alpha}$ (mod $M[I_{\alpha}])$  be a finitely solvable system of
congruences in $M$, where $x_{\alpha}\in M$ and  $I_{\alpha}$ is a left ideal of $R$ for each $\alpha\in \Lambda$.
 Consider the following system of equations. $$r_{\alpha j}x=m_{\alpha j},~  r_{\alpha j}\in I_{\alpha}, ~m_{\alpha j}=r_{\alpha j}x_{\alpha}\in M. \qquad (*)$$  Since  $x\equiv x_{\alpha}$ (mod $M[I_{\alpha}])$ is finitely solvable, so is the system of equations $(*)$, and by hypothesis it has a global solution $z\in M$. So, $z$ is a solution of $x\equiv x_{\alpha}$ (mod $M[I_{\alpha}])$.

(3)$\Leftrightarrow  $(4) is clear.
\end{proof}

Let $\mathcal{S}$ be a class of finitely presented left $R$-modules. We say that an exact sequence of left $R$-modules is $\mathcal{S}$-{\it pure} if each module of $\mathcal{S}$ is projective relatively to it. Each left $R$-module which is injective relatively to each $\mathcal{S}$-pure exact sequence is said to be $\mathcal{S}$-{\it pure injective}. When $\mathcal{S}$ contains all finitely presented left $R$-modules
we say respectively "pure exact sequence" and "pure-injective module". And, when $\mathcal{S}$ contains all finitely presented cyclic left $R$-modules
we say respectively "($\aleph_0,1$)-pure exact sequence" and "($\aleph_0,1$)-pure-injective module". For any class $\mathcal{S}$ of finitely presented left $R$-modules, each pure-exact sequence of left modules is $\mathcal{S}$-pure exact, whence each $\mathcal{S}$-pure injective left $R$-module is pure injective. So, we prove the following results by using Proposition~\ref{P:first character}(3) and \cite[Theorem 1.35(d)]{Fac98} for the first result.

\begin{example} \label{E:exe} Let $R$ be a ring. Then:
\begin{itemize}
\item[(i)] For each class $\mathcal{S}$ of   finitely presented left $R$-modules, every $\mathcal{S}$-pure injective module  is semi-compact. 
\item[(ii)] If $R$ is a domain, then every torsion-free (and so every flat) left $R$-module is semi-compact.
\end{itemize}
\end{example}

\begin{remark} \label{R:f-injective}  \textnormal{Let $R$ be a ring. Recall that a left $R$-module $M$ is  {\it
semi-injective} (or f-injective) if for each finitely
generated left ideal $I$, every $R$-homomorphism
$f:I\rightarrow M$ can be extended to an $R$-homomorphism from $R$
into $M$ (see \cite{Mat85}). By Proposition \ref{P:first character}, a left $R$-module $M$ is injective if and only if
$M$ is semi-injective and semi-compact. Since every direct sum of semi-injective left $R$-modules is semi-injective, Bass's theorem implies that
 every direct sum of semi-compact left $R$-modules is semi-compact if and only if $R$ is left Noetherian.}
\end{remark}

\medskip

Recall that a submodule $A$ of $B$ is {\it pure} (resp. ($\aleph_{0},1$)-{\it pure}) if and only if
the exact sequence $ 0\longrightarrow A
\hookrightarrow B \stackrel{g}\longrightarrow
B/A\longrightarrow 0$  is pure (resp. ($\aleph_{0},1$)-pure). 
In this case we say that   $B$ is a {\it pure extension } (resp.  ($\aleph_{0},1$)-{\it pure extension}) of $A$. It is known that $A$ is a ($\aleph_{0},1$)-pure submodule of $B$
 if and only if for each $n\in\mathbb{N}$, any system of equations $r_j x=
a_j\in A$ ($r_j\in R$, $1\leq j\leq n$) is   solvable in
$A$
whenever  it is solvable in $B$
(see \cite{War69}, \cite{Fac98}, \cite{Cou11}).
As in \cite{Azu87}, a left $R$-module $B$ is called  a {\it single
extension} of $M$ if the factor module $B/M$ is  cyclic, i.e.,
there is a cyclic submodule $A$ of $B$ such that
$B=A+M$. We say that $B$ is a {\it single pure extension} (resp. {\it single} ($\aleph_{0},1$)-{\it pure extension}) of $M$, if
 $B$ is a pure (resp. ($\aleph_{0},1$)-pure) extension and a single extension of $M$.

\begin{lemma}\label{L:upper} 
Let $M$ be a left $R$-module and $$r_jx=m_j, j\in J, r_j\in R, m_j\in M \qquad (*)$$ be a finitely solvable  system of equations in $M$.
 Then there exists
a singly pure extension  $B$ of $M$ such that the system of equations $(*)$ has a solution $b\in B$. 
\end{lemma}
\begin{proof}
Set $B=(M\oplus F)/S$, where $F$ is the free module
with the basis $\{x\}$ and $S$ is the submodule of
$M\oplus F$ generated by $\{(m_j, - r_jx) ~|~j\in
J\}$. Obviously,  $$S=\{(\sum_{k=1}^{n}z_km_k,-\sum_{k=1}^{n} z_{k}r_k x)~|~ n\in\mathbb{N},~ z_k\in R\}.$$
 Clearly,  the map $\alpha: M\rightarrow B$
defined by    $\alpha(m)=(m,0)+S$ ($m\in M$) is an
$R$-homomorphism. We claim that $\alpha$ is a monomorphism. To
see this, let $\alpha(m)=(m,0)+S=0$ for some $m\in M$. Then
$\displaystyle{m=\sum_{k=1}^{n}z_km_k}$ and $\displaystyle{\sum_{k=1}^{n} z_{k} r_{k}x=0}$ for
some $z_1,\ldots, z_n\in R$. Since the system of  equations $(*)$
is compatible, we conclude that  $m=0$ and so $\alpha$ is a
monomorphism. One can easily  see that $b=(0,x)+S\in B$ is a solution of the system of equations
$ r_jX=(m_j,0)+S\in \alpha(M)$. We claim that $\alpha(M)$
is a pure submodule of $M$. Let
$$\sum _{k=1}^{n}c_{lk}y_k=(m'_{l},0)+S\in \alpha(M), ~1\leq l\leq w,  c_{lk}\in R, m'_{l}\in M \qquad (**)$$ be a system of
equations with the solution $\{(a_k,t_{k}x)+S\}_{k=1}^{n}\subseteq B$, where  $a_{k}\in
M$ and $t_k\in R$. Then
$\displaystyle{(\sum _{k=1}^nc_{lk}a_{k}-m'_{l},-\sum_{k=1}^n c_{lk}t_{k}x)\in S}$ for each $1\leq l\leq w$. Therefore, for
each $l$ ($1\leq l\leq w$), there exist $n_l\in \mathbb{N}$,
$z_{l1},\ldots, z_{ln_l}\in R$ such that
$$\sum _{k=1}^nc_{lk}a_{k}-m'_{l}=\sum_{s=1}^{n_l}z_{ls}m_{ls} \qquad (1)$$ and
 $$\sum_{k=1}^nc_{lk}t_{k}=-\sum_{s=1}^{n_l}z_{ls}r_{ls} \qquad (2).$$ Since the system of equations $(*)$ is finitely solvable,
 there exists $m'\in M$ such that $r_{ls}m'=m_{ls}$ for some finite subset  $\{ls\}\subset J$.
 In view of $(1)$ and $(2)$ we conclude that $$\sum_{k=1}^nc_{lk}t_{k}m'=-\sum_{s=1}^{n_l}z_{ls}r_{ls}m'=-\sum_{s=1}^{n_l}z_{ls}m_{ls}=-(\sum _{k=1}^nc_{lk}a_{k}-m'_{l}).$$ Therefore, $\displaystyle{\sum _{k=1}^nc_{lk}(a_{k}-t_{k}m')=m'_{l}}$. Thus $\{(a_k-t_{k}m',0)+S\}_{k=1}^{n}\subseteq \alpha(M)$
 is a solution of the system $(**)$.
It means that $\alpha(M)$ is a pure submodule of $B$.
\end{proof}

 As in \cite{Azu87}, a left $R$-module  $M$ is  {\it singly split} in $B$ if, for every
submodule $A$ of $B$ which is a single extension of $M$, $M$ is a direct
summand of $A$, and $M$ is said to be {\it singly pure-injective} if $M$ is singly split in any pure extension of itself.  Recall that an exact sequence $\varepsilon: 0\rightarrow A
\rightarrow B \rightarrow
C\rightarrow 0$ of left $R$-modules is  {\it  cyclically pure } if every cyclic left $R$-module has the
projective property relative to $\varepsilon$ (see \cite{Sim87}).
Proposition~\ref{P:first character} leads us to obtain the following characterizations of semi-compact left  $R$-modules.

\begin{theorem} \label{T:character} 
  Let $M$ be a left $R$-module. Then the following statements are equivalent: 
\begin{enumerate}
\item $M$ is  semi-compact.
\item $M$ has the injective property relative to every ($\aleph_0,1$)-pure exact sequence $0\rightarrow  A\rightarrow  B\rightarrow  C \rightarrow 0$
where $C$ is a cyclic left $R$-module.
\item $M$ is a direct summand of every single ($\aleph_0,1$)-pure extension.
\item Every pure extension of $M$ is cyclically pure.
\item $M$ is a direct summand of every module $B$ if $B$ contains $M$ as an ($\aleph_0,1$)-pure submodule and if $B/M$ is a direct summand of a direct sum of cyclic left $R$-modules.
\item $M$ has the injective property relative to every ($\aleph_0,1$)-pure exact sequence $0\rightarrow  A\rightarrow  B\rightarrow  C \rightarrow 0$
where $C$ is a direct summand of a direct sum of cyclic left $R$-modules.
\item $M$ is singly pure-injective.
\end{enumerate}
\end{theorem}
\begin{proof}
(1)$\Rightarrow$(2). Since $C$ is cyclic then $B=A+Rb$ for some $b\in B$. Let  $f:A\rightarrow M$ be a homomorphism of left $R$-modules. Let $\{r_j\}_{j\in J}\subseteq R$ be the set of all elements of $R$
such that $r_jb\in A$. Since $A$ is an ($\aleph_0,1$)-pure submodule of $B$, the system of equations $r_jx=f(r_jb)\in M$ is finitely solvable and by hypothesis it has a solution $m\in M$. Obviously, $\phi:B\rightarrow M$ defined by $\phi(a+sb)=f(a)+sm$  for each $a\in A$ and $s\in R$ is an extension of $f$.

(2)$\Rightarrow$(3) is clear.

(3)$\Rightarrow$(4). Let $B$ be a  pure extension of $M$ and
$A\subseteq B$ be a singly extension of $M$. Then $M$ is an ($\aleph_0,1$)-pure submodule of $A$. By hypothesis,  $M$ is a summand of $A$ and so $M$ is singly split in $B$. Therefore, \cite[Theorem 3]{Azu87} implies that  the sequence $0\rightarrow  M\rightarrow  B\rightarrow  B/M \rightarrow 0$ is cyclically pure.

(4)$\Rightarrow$(5). It is well-known that every direct summand of a direct sum of cyclic modules has the projective property relative to each cyclically pure exact sequence.

(5)$\Rightarrow$(6). Let $f: A\rightarrow M$ be a homomorphism and $u: A\rightarrow B$ the inclusion map. We consider the following pushout diagram: 
\[\begin{CD}
 A @>u>> B \\
 @V{f}VV  @V{g}VV \\
  M @>v>>  T
\end{CD}\]
By   \cite[33.4(2)]{Wis91} $v$ is an ($\aleph_0,1$)-pure monomorphism. Since coker $v\cong C$ then $v$ is a split monomorphism. So, $f$ extends to a homomorphism from $B$ into $M$.

(6)$\Rightarrow$(3) is clear.

(3)$\Rightarrow$(1). Let $r_jx=m_j, j\in J, r_j\in R, m_j\in M$,
be a finitely solvable system of equations in $M$. By Lemma \ref{L:upper}, there exists
 a single pure extension $B$ of $M$ such that the system of equations $(*)$ contains a solution $b\in B$. By hypothesis,  $M$ is a direct summand of $B$.
Thus there exists a submodule $A$ of $B$ such
that $B=M\oplus A$. Therefore,  there exist
$m\in M$ and $a\in A$ such that $b=m+a$. Since $r_{j}b=m_j$ for each $j\in J$, we conclude that $r_{j}m-r_{j}a=m_j$. Thus $r_ja=r_{j}m-m_j\in A\cap
M=0$. 

(4)$\Leftrightarrow$(7) by \cite[Theorem 10]{Azu87}.
 \end{proof}

\begin{remark} \label{R:cyclically pure}
\textnormal{ Recall that a submodule $A$ of a left $R$-module $B$ is {\it cyclically pure} if and only if
the exact sequence $ 0\longrightarrow A
\hookrightarrow B \stackrel{g}\longrightarrow
B/A\longrightarrow 0$  is cyclically pure. By \cite[Proposition 1.2]{GrHi09}, $A$ is a cyclically pure submodule of $B$
if and only if for each index set $J$, any system of equations $r_j x=
a_j\in A$ ($r_j\in R$, $j\in J$) is   solvable in $A$ whenever  it is solvable in $B$.}
\end{remark}

\begin{corollary} \label{C:sum of semicompact}
Let $R$ be a  ring and $\{M_i\}_{i\in I}$ be a set of semi-compact left $R$-modules. Then $\oplus_{i\in I} M_i$ is semi-compact if and only if
$\oplus_{i\in I} M_i$  is a cyclically pure submodule of $\prod_{i\in I}M_i $.
\end{corollary}

\begin{proof}($\Rightarrow$) is clear by Theorem \ref{T:character}.\\
($\Leftarrow$). It is sufficient to show that  every finitely solvable system of equations of the form  $r_j x=
a_j\in \oplus_{i\in I} M_i$ ($r_j\in R$, $j\in J$) is   solvable in
$A$. Since $\prod_{i\in I}M_i $ is semi-compact, there exists an element $b\in \prod_{i\in I}M_i$ such that $r_jb=a_j$ for each $j\in J$. Since
$\oplus_{i\in I} M_i$ is a cyclically pure submodule of  $\prod_{i\in I}M_i$, by Remark \ref{R:cyclically pure}, there is an element $a\in \oplus_{i\in I} M_i$ such that
$r_ja=a_j$ for each $j\in J$.
\end{proof}

From  Theorem~\ref{T:character} and Remark~\ref{R:f-injective}, we deduce the following corollary.

\begin{corollary} \label{C:Noetherian} 
  Let $R$ be a ring. Then the following statements are equivalent: 
\begin{enumerate}
\item $R$ is left Noetherian.
\item Every left $R$-module is semi-compact.
\item Every direct sum of semi-compact left $R$-modules is semi-compact.
\end{enumerate}
\end{corollary}

\begin{proof}  (1)$\Rightarrow$(2). Since every cyclic left $R$-module is finitely presented and so pure projective, so, 
by Theorem~\ref{T:character} (3), every $R$-module is semi-compact.

(2)$\Rightarrow$(3) is obvious and (3)$\Rightarrow$(1) holds by Remark~\ref{R:f-injective}.
\end{proof}

A ring $R$ is called left {\it pure-semisimple} if each left $R$-module is pure-injective. Recall that a ring $R$ is {\it left perfect} if each flat left $R$-module is projective.

Theorem~\ref{T:character}  implies that there exists a semi-compact left $R$-module which is not pure injective.
It is easy to see that a left Noetherian ring $R$ is pure-semisimple if and only if every semi-compact left $R$-module is pure injective.
 In the next theorem we show that  a domain $R$ is a division ring if and only if the
 class of semi-compact $R$-modules and the class of pure injective $R$-modules coincide.

\begin{theorem} \label{T:Division} 
Let $R$ be a domain (not necessarily commutative). Then $R$ is a division ring if and only if every semi-compact
$R$-module is pure injective.
\end{theorem}

\begin{proof}
If $R$ is a division ring, it is obvious that each semi-compact module is pure injective. Conversely, let $M$ be a  flat left $R$-module. Then there exists a
free left $R$-module $F$ and submodule $K$ of $F$ such that
$\varepsilon: 0\rightarrow K \rightarrow F \rightarrow
M\rightarrow 0$ is pure exact. Since $R$ is domain, $F$ and
$K$ are torsion-free and so semi-compact. By hypothesis,
$K$ is pure injective and $\varepsilon$ splits. Therefore, $M$ is projective and so $R$ is left perfect.
 This implies that $R$ is a division ring.
 \end{proof}

Recall that a ring $R$ is {\it von-Neumann regular} if for each element $a$ of $R$ there exists $b\in R$ such that $a=aba$. It is equivalent to the fact that
every finitely generated left ideal is a summand of $R$. Also, it is known that a ring $R$ is von-Neumann regular if and only if every pure injective left $R$-module is injective. In the next theorem we have  another characterization of von-Neumann regular rings.

\begin{theorem} \label{T:Von} 
  Let $R$ be a ring. Then the following statements are equivalent: 
\begin{enumerate}
\item $R$ is von-Neumann regular.
\item Every pure injective left $R$-module is  injective.
\item Every semi-compact left $R$-module is  injective.
\end{enumerate}
\end{theorem}
\begin{proof}  (3)$\Rightarrow$(2) is obvious.

(2)$\Rightarrow$(1). Each left module is semi-injective because it is a pure submodule of a pure injective module which is injective. So, $R$ is von Neumann regular since each left module is semi-injective.

(1)$\Rightarrow$(3). In this case each left $R$-module is semi-injective. So, by \cite[Lemma 5.5]{Mat85} a left $R$-module is injective if and only if it is semi-compact.  
\end{proof}

\bigskip

\section{Rings whose flat modules are semi-compact}\label{S:flat}

\begin{definition} \label{D:def sigma}
\textnormal{Let $R$ be a ring and $M$ a left $R$-module. We say that $M$ is {\it $\Sigma$-semi-compact} if all direct sums of copies of $M$ are semi-compact.
By Example~\ref{E:exe}(2) every torsion-free module over an integral domain is $\Sigma$-semi-compact.}
\end{definition}

Faith in \cite{Fai66}, proved that an injective  left $R$-module $M$ is $\Sigma$-injective if and only if
$R$ satisfies the a.c.c. on the left ideals in $\mathcal{A}_l(R ,M)$ (equivalently, $M$ satisfies the d.c.c. on the subgroups in $\mathcal{A}_r(R, M)$).
We need the following proposition  of \cite{Fai66} to characterize  $\Sigma$-semi-compact left $R$-modules.

\begin{proposition} \label{P:faith proposition} \textnormal{\cite[Proposition 1]{Fai66}}
Let $M$ be a left $R$-module. Then $\mathcal{A}_l(R, M)$ satisfies the a.c.c., equivalently, $M$ satisfies the d.c.c. on the subgroups in $\mathcal{A}_r(R, M)$, if and only if  for each left ideal $I$ of $R$, there exists a
finitely generated subideal $I_1$ such that $M[I]=M[I_1]$.
\end{proposition}

\begin{theorem} \label{T:sigma semicompact}
Let $R$ be a ring and $M$ a left $R$-module. Then the following statements are equivalent: 
\begin{enumerate}
\item $M^{(\mathbb{N})}$ is semi-compact;
\item $R$ satisfies the a.c.c. on the left ideals in $\mathcal{A}_l(R, M)~ (M$ satisfies the d.c.c. on the  subgroups in $\mathcal{A}_r(R, M))$;
\item $M$ satisfies the d.c.c. on the  subgroups of $M$ which are annihilators of finitely generated left ideals of $R$;
\item $M$ is $\Sigma$-semi-compact.
\end{enumerate}
\end{theorem}
\begin{proof}
(1)$\Rightarrow$(2). Let $M[I_1]\supset M[I_2]\supset M[I_3]\ldots$ be a strictly descending chain, where $I_n$ is a left ideal for each integer $n\geq 1$, and let $y_n\in M[I_n]\setminus M[I_{n+1}]$.
Obviously, $x\equiv x_i$ (mod $M^{(\mathbb{N})}[I_i]=M[I_i]^{({\mathbb{N}})}$ ) is finitely solvable where $x_{i}=(y_1,\ldots,y_i,0,0,\ldots)$.
So it has a simultaneous solution in $M^{(\mathbb{N})}$ since $M^{(\mathbb{N})}$ is semi-compact.
But each
$a=(s_1,s_2,\ldots,s_t,0,0,\ldots)\in M^{(\mathbb{N})}$ cannot be a  solution of the above
system, since $a-x_{t+2}=(s_1-y_1,\ldots,s_t-y_t,y_{t+1},y_{t+2})\notin M[I_{t+2}]^{({\mathbb{N}})}$, a
contradiction. 

(2)$\Rightarrow$(3) is clear.

(3)$\Rightarrow$(4). First we show that $M$ is semi-compact. Let $x\equiv x_i$ (mod $M[I_{\alpha}]$ ), where $\alpha\in \Lambda$, be a finitely solvable system. By Proposition \ref{P:faith proposition} we may assume that $I_{\alpha}$ is finitely generated for each $\alpha\in\Lambda$. There exists $V=M[I_{\alpha1}]\cap\ldots\cap M[I_{\alpha n}]$ which is minimal among the set of all finite intersections of the  $M[I_{\alpha}]$ for $\alpha\in \Lambda$.
Obviously,  $V\subseteq M[I_{\alpha}]$ for each  $\alpha\in \Lambda$. Let $y-x_{\alpha i}\in M[I_{\alpha i}]$ for each $1\leq i\leq n$. Let
 $y_{\beta}$ be a solution of $x\equiv x_{\alpha i}$ (mod $M[I_{\alpha i}])$ and $x\equiv x_{\beta}$ (mod $M[I_{\beta}])$  where $1\leq i\leq n$.
  It is easy to check that $y-y_{\beta}\in V\subseteq M[I_{\beta}]$. Thus $y-x_{\beta}=y-y_{\beta}+y_{\beta}-x_{\beta}\in M[I_{\beta}]$. Therefore, $y$ is a simultaneous solution in $M$. So,
 $M$ is semi-compact. Let $J$ is an index set, $I$  a  left ideal of $R$ and $h:I\rightarrow M^{(J)}$
a homomorphism such that, for any finitely generated left ideal $I_0$ of $I$, there exists $m_0\in M^{(J)}$ such that $h(r_0)=r_0m_0$ for each $r_0\in I_0$.
Let $I_1=Rs_1+\ldots+Rs_n$ be the finitely generated subideal of $I$ given by Proposition \ref{P:faith proposition} such that $M[I]=M[I_1]$. Therefore, there exists
$m_1=(m_{1j})_{j\in J}\in M^{(J)}$  such that $h(r_1)=r_1m_1$ for each $r_1\in I_1$. Since $M^J$ is semi-compact, there exists an element
$m'=(m'_j)_{j\in J} \in M^J$ such that $h(r)=rm'$ for each $r\in I$.   Thus for each $j\in J$, $m'_j-m_{1j}\in M[I_1]=M[I]$. We conclude that
$h(r)=rm_1$ for each $r\in I$.

(4)$\Rightarrow$(1) is clear.
\end{proof}

\begin{corollary} \label{C:sigma semicompact submodule}
Every submodule of a $\Sigma$-semi-compact module is $\Sigma$-semi-compact.
\end{corollary}

\begin{corollary} \label{C:factor of sigma semicompact}
Let $R$ be a  ring and $I$ a two-sided ideal of $R$. If $M$ is a $\Sigma$-semi-compact left $R$-module, then so is $M/M[I]$.
\end{corollary}
\begin{proof}
It is clear by Theorem \ref{T:sigma semicompact} and corollary \ref{C:sigma semicompact submodule}.
\end{proof}

\begin{corollary} \label{C:sigma semicompact ring}
Let $R$ be a ring. Then the following statements are equivalent:
\begin{enumerate}
\item $R$ satisfies the  a.c.c. on the left annulets.
\item $R$  satisfies the  d.c.c. on the right annulets.
\item $R$ is $\Sigma$-semi-compact as left $R$-module.
\end{enumerate}
\end{corollary}

\begin{corollary} \label{C:sigma}
Let $M$ be a semi-injective left $R$-module. Then $M$ is $\Sigma$-injective if and only  if  it is $\Sigma$-semi-compact.
\end{corollary}

The following theorem is a generalization of Corollary \ref{C:Noetherian}.

\begin{theorem} \label{T:Noeth-pproj}
Let $R$ be a ring. Then every pure projective left $R$-module is semi-compact if and only if $R$ is left Noetherian.
\end{theorem}
\begin{proof}
Let $C$ be a cyclic left $R$-module. By \cite[Theorem 33.5]{Wis91}, there exists a pure exact sequence $\varepsilon: 0\longrightarrow K \longrightarrow P \longrightarrow
C\longrightarrow 0$ where $P$ is a pure projective left module. By Corollary \ref{C:sigma semicompact submodule}, $K$ is $\Sigma$-semi-compact. Therefore, $\varepsilon$ splits and so $C$ is a direct summand of $P$. Since $C$ is cyclic, it is a direct summand of a finite direct sum of finitely presented $R$-modules. It follows that $C$ is finitely presented. Hence $R$ is left Noetherian. The converse is clear by Corollary \ref{C:Noetherian}.
\end{proof}

\begin{theorem} \label{T:flat is  semicompact}
Let $R$ be a  ring. Then the following statements are equivalent: 
\begin{enumerate}
\item Every flat left $R$-module is semi-compact.
\item $R$ is $\Sigma$-semi-compact as left $R$-module.
\item $R$  satisfies the  a.c.c. on the left annulets ($R$ satisfies the  d.c.c. on the right  annulets).
\end{enumerate}
\end{theorem}

\begin{proof}
 (1)$\Rightarrow$(2) is clear.

 (2)$\Leftrightarrow$(3) by Corollary \ref{C:sigma semicompact ring}.

 (2)$\Rightarrow$(1). Let $M$ be a flat left $R$-module. Then $M=F/K$ where $F$ is a free $R$-module and
 $K$ is a pure submodule of $F$. Suppose that $M$ is
not $\Sigma$-semi-compact.   Then by Theorem \ref{T:sigma semicompact}, there exists a strict descending chain  $M[I_1]\supset M[I_2]\supset\ldots$. By Proposition \ref{P:faith proposition}, we may assume that
for each $i\in \mathbb{N}$, $I_i$ is finitely generated. For each $i\in \mathbb{N}$, let $a_i+K\in M[I_{i}]\setminus M[I_{i+1}]$  and let $r_{i,1},\ldots r_{i,m_i}$ be
 generators  of $I_{i}$ where $m_i\in \mathbb{N}$. Therefore, $r_{i,j}a_i\in K$ for  $j=1,\ldots,m_i$. Since $K$ is pure submodule of $F$, there exist $k_i\in K$ such that $I_i(a_i-k_i)=0$ for each $i\in \mathbb{N}$. One can easily see that $I_{i+1}(a_i-k_i)\neq 0$ for each $i\in \mathbb{N}$. Thus we get the strict descending chain  $F[I_1]\supset F[I_2]\supset\ldots\supset F[I_n]\supset\ldots$. This contradicts   that $F$ is $\Sigma$-semi-compact.
 \end{proof}

In \cite{Bjo70}, Bj\"{o}rk proved that a left semi-injective ring $R$ is quasi-Frobenius if and only if it satisfies the  a.c.c. on the left annulets.
 Therefore we have the following evident corollary by Theorem \ref{T:flat is  semicompact}.

\begin{corollary} \label{C:QF}
Let $R$ be a self left semi-injective ring. Then each  flat
left $R$-module is semi-compact if and only if  $R$ is a quasi-Frobenius ring.
\end{corollary}

\begin{proposition}
\label{P:QNoe} Let $R$ be a ring. Assume that $R$ is a subring of a left Noetherian ring $S$. Then each flat left $R$-module is semi-compact.
\end{proposition}

\begin{proof}
As left $S$-module, $S$ is $\Sigma$-semi-compact. If $A$ is a left ideal of $R$ and $A'=SA$ then it is easy to check that $S[A]=S[A']$. So $S$ is  a left $\Sigma$-semi-compact  $R$-module. It follows that so is $R$ by Corollary~\ref{C:sigma semicompact submodule}.
\end{proof}

From these last two propositions we deduce the following corollary.

\begin{corollary}
Let $R$ be a commutative ring and $Q$ its quotient ring. Assume that $Q$ is semi-injective. Then each flat $R$-module is semi-compact if and only if $Q$ is quasi-Frobenius.
\end{corollary}

\section{Finite projectivity and $\Sigma$-semi-compactness}\label{S:fproj}

As in \cite{Azu87}, a left  $R$-module $M$ is called {\it finitely projective} (respectively {\it singly projective}) if any homomorphism from a finitely generated (respectively cyclic) left $R$-module into $M$ factors through a free left $R$-module. If $m,n$ are positive integers, a right $R$-module is said to be {\it ($m,n$)-flat} if, for each $n$-generated left submodule $K$ of $R^m$, the homomorphism $M\otimes_RK\rightarrow M\otimes_RR^m$ deduced of the inclusion map is injective. We say that $M$ is ($\aleph_0,1)$)-flat if it is ($m,1)$)-flat for each integer $m>0$. In \cite[Theorem 5]{She91} Shenglin proved that every flat left module is singly projective if, for each descending chain of finitely generated right ideals $I_1\supseteq I_2\supseteq I_3\supseteq\dots$, the ascending chain ${}^{\bot}I_1\subseteq {}^{\bot}I_2\subseteq {}^{\bot}I_3\subseteq\dots$ terminates. Therefore by Corollary~\ref{C:sigma semicompact ring}(1), we deduce that every flat left module is singly projective if $R$ is $\Sigma$-semi-compact as left module. The following proposition generalizes this result.

\begin{proposition}\label{P:(N_0,1)-flat}
Let $R$ be a ring which is $\Sigma$-semi-compact as  left $R$-module. Then each ($\aleph_0,1$)-flat left $R$-module is singly projective.
\end{proposition}

\begin{proof}
Let $M$ be a ($\aleph_0,1$)-flat left $R$-module and let $\pi:F\rightarrow M$ be an epimorphism with $F$ a free left $R$-module. Let $C$ be a left cyclic module generated by $c$, let $f:C\rightarrow M$ be a homomorphism and let $A={}^{\bot}\{c\}$. Since $F$ is $\Sigma$-semi-compact, so by Proposition~\ref{P:faith proposition}  there exists a finitely generated left subideal $B$ of $A$ such that $F[A]=F[B]$. Let $g:R/B\rightarrow M$ be the homomorphism defined by $g(1+B)=f(c)$. Since $M$ is ($\aleph_0,1$)-flat, so, by \cite[Proposition 4.1 and Theorem 1.1]{Cou11}, there exists a homomorphism $h:R/B\rightarrow F$ such that $g=\pi\circ h$. Let $x=h(1+B)$. Then $x\in F[B]=F[A]$. So, the homomorphism $\phi:C\rightarrow F$ defined by $\phi(c)=x$ satisfies $f=\phi\circ\pi$.
\end{proof}

\begin{corollary} \label{C:singly-semi-compact} 
Let $R$ be a ring such that $R^{\mathbb{N}}$ is ($\aleph_0,1$)-flat as left module. Then the following conditions are equivalent:
\begin{enumerate}
\item $R$ is  $\Sigma$-semi-compact as left module.
\item Each ($\aleph_0,1$)-flat left $R$-module is singly projective.
\end{enumerate}
\end{corollary}

\begin{proof}
(1)$\Rightarrow$(2) follows from Proposition \ref{P:(N_0,1)-flat}.

(2)$\Rightarrow$(1). From $R^{\mathbb{N}}$ ($\aleph_0,1$)-flat and $R^{(\mathbb{N})}$ pure submodule of $R^{\mathbb{N}}$ we deduce that $R^{\mathbb{N}}/R^{(\mathbb{N})}$
is ($\aleph_0,1$)-flat. We conclude by \cite[Proposition 1]{Len76} and Corollary \ref{C:sigma semicompact ring}.
\end{proof}

An $R$-module is called {\it uniserial} if if the set of its submodules is totally ordered by inclusion. A commutative ring $R$ is a {\it chain ring} (or a valuation ring) if it is a uniserial $R$-module.

The following example satisfies the equivalent conditions of the previous corollary but the hypothesis does not hold.

\begin{example}
\textnormal{Let $R$ be a chain ring whose quotient ring $Q$ is Artinian and not reduced. We assume that $\mathrm{Spec}\ R$ is finite and $R\ne Q$. By \cite[Corollary 36]{Couch03} each ideal is countably generated. Each flat $R$-module is semi-compact by Proposition~\ref{P:QNoe} and singly projective by \cite[Theorem 7]{Coucho07}. Since $R\ne Q$, $R$ is not FP-injective as $R$-module, so, by \cite[Theorem 37(6)]{Couch03} $R^{\mathbb{N}}$ is not  $(\aleph_0,1)$-flat\footnote{Over a chain ring each $(\aleph_0,1)$-flat module is flat.}.}
\end{example}

If $R$ is a commutative ring, then we consider on $\mathrm{Spec}\ R$ the equivalence relation $\mathcal{R}$ defined by   $L\mathcal{R} L'$ if there exists a finite sequence of prime ideals $(L_k)_{1\leq k\leq n}$ such that $L=L_1,$ $L'=L_n$ and $\forall k,\ 1\leq k\leq (n-1),$ either $L_k\subseteq L_{k+1}$ or $L_k\supseteq L_{k+1}$. We denote by $\mathrm{pSpec}\ R$ the quotient space of $\mathrm{Spec}\ R$ modulo $\mathcal{R}$ and by $\lambda_R: \mathrm{Spec}\ R\rightarrow\mathrm{pSpec}\ R$ the natural map. The quasi-compactness of $\mathrm{Spec}\ R$ implies the one of $\mathrm{pSpec}\ R$, but generally $\mathrm{pSpec}\ R$  is not  $T_1$: see \cite[Propositions 6.2 and 6.3]{Laz67}.

\begin{lemma}[{\cite[Lemma 2.5]{Couc09}}]
\label{L:pure} Let $R$ be a commutative ring and let $C$ a closed subset of $\mathrm{Spec}\ R$. Then $C$ is the inverse image of a closed subset of $\mathrm{pSpec}\ R$ by $\lambda_R$ if and only if $C=V(A)$ where $A$ is a pure ideal. Moreover, in this case, $A=\cap_{P\in C}0_P$ (where $0_ P$ is the kernel of the canonical homomorphism $R\rightarrow R_ P$).
\end{lemma}

A commutative ring $R$ is called \textit{arithmetical} if $R_L$ is a chain ring for each maximal ideal $L$.

\begin{theorem} \label{T:Areth}
Let $R$ be a commutative arithmetical ring and $Q$ its quotient ring. Then the following conditions are equivalent:
\begin{enumerate}
\item $Q$ is pure-semisimple;
\item each flat $R$-module is semi-compact;
\item each flat $R$-module is finitely projective;
\item each flat $R$-module is singly projective.
\end{enumerate}
\end{theorem}

\begin{proof}
$(1)\Rightarrow (2)$ is a consequence of Proposition~\ref{P:QNoe}, $(1)\Rightarrow (3)$ follows from \cite[Theorem 7]{Coucho07}, $(2)\Rightarrow (4)$ holds by Proposition \ref{P:(N_0,1)-flat}  and $(3)\Rightarrow (4)$ is obvious.

$(4)\Rightarrow (1)$. First we show that $\mathrm{Min}\ R$ is finite. Since each prime ideal contains a unique minimal prime ideal, then each point of $\mathrm{pSpec}\ R$ is of the form $V(L)$ where $L$ is a minimal prime ideal. By Lemma~\ref{L:pure} there exists a pure ideal $A$ such that $V(L)=V(A)$. Since $R/A$ is flat, it is projective, whence $A=Re$, where $e$ is an idempotent of $R$. Hence each single subset of $\mathrm{pSpec}\ R$ is open. From the quasi-compacity of  $\mathrm{pSpec}\ R$, we deduce that $\mathrm{Min}\ R$ is finite. . Let $P$ be a maximal ideal of $R$. By using \cite[Proposition 6]{Coucho07} we get that $R_P$ satisfies $(3)$. By \cite[Theorem 33]{Coucho07} the quotient ring $Q(R_P)$ of $R_P$ is artinian. It follows that $Q(R_P)=R_L$ where $L$ is the minimal prime ideal contained in $P$. Let $s$ be an element of $R$ which does not belong to any minimal prime ideal. If $a\in R$ satisfies $sa=0$ then it is easy to check that $\dfrac{a}{1}=0$ in $R_P$ for each maximal ideal $P$. So, $a=0$ and $s$ is regular. We deduce that $Q\cong\prod_{L\in\mathrm{Min}\ R}R_L$. Hence $Q$ is pure-semisimple.
\end{proof}

The following proposition is a slight generalization of \cite[Proposition 6]{Coucho07} and the proof is similar.

\begin{proposition} \label{P:locali} 
Let $\phi: R\rightarrow S$ be a right flat epimorphism of rings. Then:
\begin{enumerate}
\item For each singly (respectively finitely) projective left
 $R$-module $M$, $S\otimes_RM$ is singly (respectively finitely) projective over $S$;
\item Let $M$ be a singly (respectively finitely) projective left $S$-module. If $\phi$ is injective then $M$ is singly (respectively finitely) projective over $R$.
\end{enumerate}
\end{proposition}

\begin{theorem} \label{T:Flat=f-pro} Let $R$ be a ring. Assume that $R$ has a right flat epimorphic extension $S$ which is von Neumann regular. Then the following conditions are equivalent:
\begin{enumerate}
\item $S$ is semisimple;
\item each flat left $R$-module is semi-compact;
\item each flat left $R$-module is finitely projective;
\item each flat left $R$-module is singly projective.
\end{enumerate}   
\end{theorem}

\begin{proof}
$(1)\Rightarrow (3)$ is an immediate consequence of \cite[Corollary 7]{She91} and $(3)\Rightarrow (4)$ is obvious. 

$(1)\Rightarrow (2)$ is an immediate consequence of Proposition~\ref{P:QNoe}, and $(2)\Rightarrow (4)$ holds by Proposition~\ref{P:(N_0,1)-flat}.

$(4)\Rightarrow (1)$. First we show that each left $S$-module $M$ is singly projective. Every left $S$-module $M$ is flat over $S$ and $R$. So, $M$ is singly projective over $R$. It follows that $M\cong S\otimes_RM$ is singly projective over $S$ by Proposition~\ref{P:locali}(1). Now let $A$ be a left ideal of $S$. Since $S/A$ is singly projective, it is projective. So, $S/A$ is finitely presented over $S$ and $A$ is a finitely generated ideal of $S$. Hence $S$ is semisimple. 
\end{proof}

\begin{corollary} \label{C:Flat=f-pro} Let $R$ be a commutative reduced ring and $Q$ its quotient ring. Assume that the space $\mathrm{Min}\ R$ of minimal prime ideals of $R$ is compact in its Zariski topology. Then the following conditions are equivalent:
\begin{enumerate}
\item $Q$ is semisimple;
\item each flat $R$-module is semi-compact;
\item each flat $R$-module is finitely projective;
\item each flat $R$-module is singly projective.
\end{enumerate}   
\end{corollary}

\begin{proof}
We use the assumption that $\mathrm{Min}\ R$ is compact. By \cite[Theorem 3.14.1]{Ste71} and \cite[Proposition 1]{Que71} $Q$ is a subring of a von Neumann regular ring $S$ such that the inclusion map $R\rightarrow S$ is a flat epimorphism ($S$ is the maximal flat epimorphic extension of $R$). If either $Q$ or $S$ is semisimple, then $Q=S$.
\end{proof} 

Given a ring $R$, a left $R$-module $M$ and $x\in M$,  the \textit{content ideal} $\mathrm{c}(x)$ of $x$ in $M$, is the intersection of all right ideals $A$ for which $x\in AM$. We say that $M$ is a \textit{content module} if $x\in\mathrm{c}(x)M,\ \forall x\in M$. We say that $M$ is {\it FP-injective} if $\mathrm{Ext}_R^1(F,M)=0$ for each finitely presented left $R$-module $F$. It is easy to see that each FP-injective module is semi-injective, but we do not know if the converse holds, except for some classes of rings.

\begin{proposition} \label{P:Perfect}
Let $R$ be a self left FP-injective ring. Then each  flat
left $R$-module is finitely projective if and only if  $R$ is left perfect.
\end{proposition}
\begin{proof}
Let $M$ be a flat left $R$-module. Since it is finitely projective, so it is FP-injective and a content module by \cite[Proposition 3(2)]{Coucho07}. We conclude that $R$ is left perfect by \cite[Theorem 2]{Coucho07}.
\end{proof}

\begin{corollary} \label{C:Flat=f-proj} Let $R$ be a ring. Assume that $R$ has a right flat epimorphic extension $S$ which is self left FP-injective. Then each flat left $R$-module is finitely projective if and only if $S$ is left perfect.
\end{corollary}
\begin{proof}
If $S$ is left perfect we conclude by \cite[Corollary 7]{She91}. Conversely, first we show that each flat left $S$-module is finitely projective. The proof is similar to that of $(4)\Rightarrow (1)$ of Theorem~\ref{T:Flat=f-pro}, and then we use the previous proposition.
\end{proof}

By \cite[Corollary 16]{ZiZi78} (a result due to Jensen) each semiprimary ring with square of the Jacobson radical zero is $\Sigma$-pure-injective (hence $\Sigma$-semi-compact) on either side. Since these rings are left and right perfect then each flat module is projective.

\begin{remark}\label{R:cond}
\textnormal{Consider the following two conditions:
\begin{enumerate}
\item each flat left $R$-module is semi-compact;
\item each flat left $R$-module is finitely projective.
\end{enumerate}
Let us observe that there are many examples of rings satisfying the two conditions. We shall see that they are not equivalent. Does the first condition imply the second?}
\end{remark}

Let $R$ be a ring which is a FP-injective  left module. Then $R$ is left perfect if and only if $R$ satisfies the second condition by Proposition~\ref{P:Perfect}. By Corollary~\ref{C:QF}, $R$ satisfies the first condition if and only if $R$ is quasi-Frobenius. It remains to give an example of a left perfect ring which is self left FP-injective and which is not quasi-Frobenius.

\begin{proposition}\label{P:perf}
Let $R$ be a local commutative ring of maximal ideal $P$ such that $P^2$ is the only  minimal non-zero ideal of $R$. Then:
\begin{enumerate}
\item $R$ is perfect and self FP-injective;
\item $R$ is quasi-Frobenius if and only if $P$ is finitely generated if and only if $R^{\mathbb{N}}$ is $(1,1)$-flat.
\end{enumerate}
\end{proposition}

\begin{proof}
$(1)$. Since $R$ is local and $P$ nilpotent ($P^3=0$), $R$ is perfect. For each $R$-module $M$ we put $M^*=\mathrm{Hom}_R(M,R)$. By \cite[Theorem 2.3]{Ja73} to show that $R$ is self FP-injective it is enough to prove that the evaluation map $\phi_M:M\rightarrow M^{**}$ is injective for each finitely presented $R$-module $M$. We consider a finitely presented module $M$. We have the following exact sequence $0\rightarrow K\xrightarrow{u} F\xrightarrow{\pi} M\rightarrow 0$ where $F$ is a free $R$-module of finite rank, $K$ a finitely generated submodule of $F$ and $u$ the inclusion map. We may assume that $K\subseteq PF$. We have the following commutative diagram with exact horizontal sequences:

\[\begin{CD}
 K @>u>> F @>\pi>>M  \\
 @V{\phi_K}VV  @V{\phi_F}VV @V{\phi_M}VV \\
  K^{**} @>u^{**}>>  F^{**} @>\pi^{**}>>  M^{**}
\end{CD}\]

\bigskip

Since $\phi_F$ is an isomorphism and $u$ a monomorphism then $\phi_K$ is injective. On the other hand, if $E=\mathrm{E}(R/P)$ then $E\cong\mathrm{E}(R)$. If $N$ is a module of finite length, denoted by $\ell(N)$, then $\ell(N)=\ell(\mathrm{Hom}_R(N,E))$ (this can be proved by induction on $\ell(N)$). We have $PK\subseteq P^2F$. So, $PK$ is a semisimple module of finite length. Since so is $K/PK$, it follows that $K$ is of finite length too. From $K^*\subseteq\mathrm{Hom}_R(K,E)$ we deduce that $\ell(K^*)\leq \ell(K)$. In the same way we get that $\ell(K^{**})\leq\ell(K)$. Whence $\ell(K^{**})=\ell(K)$, $\phi_K$ is an isomorphism and $u^{**}$ is injective. From snake Lemma we deduce that $\phi_M$ is a monomorphism.

$(2)$. The first equivalence is obvious and the second is a consequence of Corollaries \ref{C:QF} and \ref{C:singly-semi-compact}, and \cite[Theorem 4.11]{Cou11}.
\end{proof}

\begin{example}
Let $K$ be a field, $\Lambda$ an index set and $\alpha\in\Lambda$. Let $R$ be the factor ring of the polynomial ring $K[X_{\lambda}\mid \lambda\in\Lambda]$ modulo the ideal generated by \[\{X_{\lambda}^2-X_{\alpha}^2\mid \lambda\in\Lambda\}\cup\{X_{\lambda}X_{\mu}\mid \lambda,\mu\in\Lambda, \lambda\ne\mu\}.\] Then $R$ satisfies the assumptions of Proposition~\ref{P:perf}. Consequently, if $\Lambda$ is not a finite set then $R$ verifies the second condition of Remark~\ref{R:cond} but not the first.
\end{example}

\section{Semi-compactness and pure-injectivity}\label{S:pure-injectivity}

By Example~\ref{E:exe}(1) each pure-injective module is semi-compact. By Theorem~\ref{T:Von} the converse holds over every von Neumann regular ring. From Corollary~\ref{C:Noetherian} we deduce the following:
\begin{corollary}
\label{C:NoPss} Let $R$ be a left Noetherian ring. Then each semi-compact left $R$-module is pure-injective if and only if $R$ is left pure-semisimple.
\end{corollary}

Now, we investigate rings for which each semi-compact left module is pure-injective. We shall give a partial answer.

Recall that a left $R$-module $M$ is {\it cotorsion} if $\mathrm{Ext}_R^1(F,M)=0$ for each flat left $R$-module $F$. It is easy to check that every pure-injective module is cotorsion, and, by \cite[Proposition 3.3.1]{Wu96} a ring $R$ is left perfect if and only if each left $R$-module is cotorsion. So, if $R$ is left Artinian, then each left $R$-module is semi-compact and cotorsion, but each left module is pure-injective if and only $R$ is left pure-semisimple.

The following theorem completes Theorem~\ref{T:Von}.

\begin{theorem}
\label{T:Von1} For any ring $R$ the following conditions are equivalent:
\begin{enumerate}
\item $R$ is von Neumann regular;
\item each cotorsion left (right) $R$-module is injective.
\end{enumerate}
\end{theorem}

\begin{proof}
$(1)\Rightarrow (2)$. Since each left $R$-module is flat, so each cotorsion left module is injective.

$(2)\Rightarrow (1)$. Let $M$ be a left $R$-module. We shall prove that $M$ is flat. By \cite[Theorem 3]{BiEBEn01} and \cite[Lemma 2.1.1]{Wu96}, there exists an exact sequence \[0\rightarrow K\rightarrow F\rightarrow M\rightarrow 0,\] where $F$ is flat and $K$ cotorsion. Since $K$ is injective, the sequence splits and we deduce that $M$ is flat. 
\end{proof}

\begin{proposition}\label{P:sc=pi}
Let $R$ be a commutative ring. Then: 
\begin{enumerate}
\item if each semi-compact $R$-module is cotorsion (repectively pure-injective) then,
for each multiplicative subset $S$ of $R$, each semi-compact $S^{-1}R$-module is cotorsion (respectively pure-injective).
\item if each semi-compact $R$-module is pure-injective then each prime ideal of $R$ is maximal.
\end{enumerate}
\end{proposition}

\begin{proof}
$(1)$. 
Let $M$ be a semi-compact $S^{-1}R$-module. Then $M$ is semi-compact over $R$ too. It follows that $M$ is cotorsion (respectively pure-injective) as $R$-module. We easily check that it satisfies this property as $S^{-1}R$-module.

$(2)$. We apply Theorem~\ref{T:Division} to $R/L$ where $L$ is a prime ideal. 
\end{proof}

\begin{theorem}
\label{T:reduced} Let $R$ be a commutative reduced ring. Then the following conditions are equivalent:
\begin{enumerate}
\item $R$ is von Neumann regular;
\item each semi-compact $R$-module is pure-injective.
 \end{enumerate}
\end{theorem}

\begin{proof}
It easy to prove that $(1)\Rightarrow (2)$.

$(2)\Rightarrow (1)$. By Proposition~\ref{P:sc=pi}(2), each prime ideal is maximal. It follows that $R_P$ is a field for each maximal ideal $P$ of $R$. Hence $R$ is von Neumann regular.
\end{proof}

\begin{proposition}\label{P:nonidemp}
Let $R$ be a commutative local ring of maximal ideal $P$. Assume that $P\ne P^2$. Then $R$ is pure-semisimple if each semi-compact $R$-module is pure-injective.
\end{proposition}
\begin{proof}
By Proposition~\ref{P:sc=pi}, $P$ is the only prime ideal of $R$. So, it suffices to show that $\dim_{R/P}P/P^2=1$. By way of contradiction suppose that $\dim_{R/P}P/P^2>1$. After replacing $R$ by a suitable factor, we may assume that $1<\dim_{R/P}P/P^2<\infty$. So, $R$ is Artinian but not pure-semisimple. Then each $R$-module is semi-compact (and cotorsion) but there exists a module which is not pure-injective, whence a contradiction.
\end{proof}

\begin{proposition}\label{P:Jacregular}
Let $R$ be ring and $J$ its Jacobson radical. Assume that $R/J$ is von Neumann regular and $J$ nilpotent. Then each semi-compact left (right) $R$-module is cotorsion.
\end{proposition}

\begin{proof}
If $J=0$, then we apply Theorem~\ref{T:Von}. Let $n$ be the smallest integer satisfying $J^n=0$. We proceed by induction on $n$. Let $M$ be a semi-compact module. We consider the following exact sequence:
\[0\rightarrow M[J^p]\rightarrow M[J^{p+1}]\rightarrow M[J^{p+1}]/M[J^p]\rightarrow 0.\]
We assume that the theorem holds if $n=p$.
For any proper two-sided ideal $A$, a left $R/A$-module is cotorsion as $R$-module if so is as $R/A$-module (see \cite[Proposition 3.3.3]{Wu96}). On the other hand $M[J^p]=M[J^{p+1}][J^p]$. By using Proposition~\ref{P:submodule} and the fact that the class of cortorsion modules is closed by extension, we get that $M[J^{p+1}]$ is cortorsion, because $M[J^p]$ and $M[J^{p+1}]/M[J^p]$ are left modules over $R/J^p$ and they are semi-compact by Proposition~\ref{P:submodule}.  
\end{proof}

\begin{corollary}
\label{C:nilpotent} Let $R$ be a commutative ring and $N$ its nilradical. Then:
\begin{enumerate}
\item  $R_P$ is perfect for each maximal ideal $P$ if $N$ is T-nilpotent and if every semi-compact $R$-module is cotorsion.
\item  $R_P$ is pure-semisimple for each maximal ideal $P$ if $N$ is T-nilpotent and if every semi-compact $R$-module is pure-injective.
\item Every semi-compact $R$-module is cotorsion if $N$ is nilpotent and if each prime ideal is maximal.
\end{enumerate}
\end{corollary}

\begin{proof}
(1). By Proposition~\ref{P:sc=pi}(2) each prime ideal is maximal. So, for each maximal ideal $P$ the Jacobson of $R_P$ is T-nilpotent, whence $R_P$ is perfect.

(2). As in (1) we prove that $R_P$ is perfect for each maximal ideal $P$. So, $PR_P\ne (PR_P)^2$. We use Proposition~\ref{P:nonidemp} to conclude.

(3). It is easy to see that $N$ is the Jacobson radical of $R$ and that $R/N$ is von Neumann regular. We conclude by Proposition~\ref{P:Jacregular}.
\end{proof}



\end{document}